\documentclass[10pt,reqno]{amsart}

\usepackage{amsmath,amssymb,amsfonts,dsfont}
\usepackage{cite,graphicx,hyperref}

\PassOptionsToPackage{linktocpage}{hyperref}

\theoremstyle{plain}
  \newtheorem{theorem}{Theorem}

\theoremstyle{remark}
  \newtheorem{remark}{Remark}

\renewcommand{\phi}{\varphi}
\newcommand{\ca}{\mathcal{A}}
\newcommand{\cb}{\mathcal{B}}
\newcommand{\cd}{\mathcal{D}}
\newcommand{\cj}{\mathcal{J}}
\newcommand{\ck}{\mathcal{K}}
\newcommand{\cl}{\mathcal{L}}
\newcommand{\cv}{\mathcal{V}}
\newcommand{\LPT}{{L_2(\mathds{T})}}
\newcommand{\LPTT}{{L_2(\mathds{T}^2)}}

\author{Konstantin A. Rybakov}

\title{On Approximate Representation of Fractional Brownian Motion}

\begin{document}

\maketitle

\textbf{Abstract.} This paper considers the orthogonal expansion of the fractional Brownian motion relative to the Legendre polynomials. Such an expansion has not only theoretical but also practical interest, since it can be applied to approximate and simulate the fractional Brownian motion in continuous time. The relations for the mean square approximation error are presented, and a comparison with the previously obtained result is carried out.

\vskip 0.5ex


\vskip 0.5ex


\makeatletter{\renewcommand*{\@makefnmark}{}
\footnotetext{Email: rkoffice@mail.ru}
\footnotetext{Citation: Rybakov, K.A. On Approximate Representation of Fractional Brownian Motion. {\em Methodol. Comput. Appl. Probab.} {\bf 2025}, {\em 27(4)}, 88. \url{https://doi.org/10.1007/s11009-025-10219-w}}\makeatother}

\renewcommand{\thefootnote}{\fnsymbol{footnote}}

\thispagestyle{empty}

\section{Introduction}\label{secIntro}

The analysis of probability models for various objects often involves not only obtaining theoretical results but also the statistical simulation. It can be used to verify theoretical results or, if such results are missing, it can be the only source of information about objects under consideration. For statistical simulation, the key is the ability to obtain realizations of scalar and vector random variables as well as random processes~\cite{MihVoi_06}.

In this paper, we study a one-parameter family of Gaussian random processes, namely the fractional Brownian motion with a parameter called the Hurst index. The aim of this paper is to present a method for its approximation and simulation in continuous time using the orthogonal expansion of this random process~\cite{Ryb_Comp25}.

The simplest method for simulation of a Gaussian random process given by a mathematical expectation and covariance function is reduced to simulation of a Gaussian random vector of increments corresponding to the nodes of a uniform time grid~\cite{MihVoi_06, Law_21}. This method is universal, i.e., suitable for any Gaussian random process. It guarantees the exact result at the grid nodes, i.e., the simulation result corresponds to the given distribution law, but it only provides an approximation of a random process in discrete time.

The fractional Brownian motion simulation can be performed using different approaches, many of which are also universal, i.e., suitable for any Gaussian random process. Their efficiency is evaluated using various criteria by which a comparison can be carried out: accuracy, computational complexity, computational speed, amount of memory used, etc. We recommend that the reader refer to frequently cited reviews of methods for the fractional Brownian motion simulation~\cite{Coe_JSS00, Die_MS02, DieMan_PEIS03, KijTam_13}. Despite quite some time passing since their publication, they have not lost their relevance. Methods described in these reviews are either applied without modification~\cite{VarBul_SPL15, CreMarMus_FF22, RegDolBen_NC23}, or they are used as the basis for more advanced methods to solve specific problems~\cite{WalWie_PRE20}.

A brief overview of various natural and applied problems whose mathematical models include the fractional Brownian motion is given in~\cite{AruWalWie_PRE20}. This overview is useful because it not only lists specific applications with references to published results but also specifies the range for the Hurst index. The fractional Brownian motion simulation is important not only in specific applications but also for verifying theoretical results that are related to different functionals depending on the fractional Brownian motion~\cite{SadWie_PRE21, KimAki_PRE22}.

The paper~\cite{Ryb_Comp25} proposes the orthogonal expansion of the fractional Brownian motion relative to the Legendre polynomials for approximation and simulation of this random process. The main result from~\cite{Ryb_Comp25} is based on two components. The first one is the integral representation of the fractional Brownian motion proposed in~\cite{DecUst_PA99}. The second one is the spectral form of mathematical description of control systems~\cite{SolSemPeshNed_79, RybSot_TAC07}. The paper~\cite{Ryb_Comp25} gives the expansion of the kernel, related to the integral representation of the fractional Brownian motion, into orthogonal series with respect to the Legendre polynomials. The matrix of expansion coefficients is formed as the product of four matrices: two matrices corresponding to multiplication operators with power functions as multipliers and two matrices corresponding to fractional integration operators.

In this paper, simpler relations for expansion coefficients are obtained. Their use reduces the computational complexity and increases the accuracy for the approximate representation of the fractional Brownian motion. The method proposed here does not replace but complements the previously obtained results. The method for finding expansion coefficients from~\cite{Ryb_Comp25} ensures computational stability when using machine arithmetic operations. The approach considered in this paper requires the use of libraries that provide floating point calculations with any given accuracy or the use of symbolic calculations, but it allows one to calculate exactly the mean square approximation error when applying both methods. We emphasize that both methods provide the fractional Brownian motion approximation in continuous time. They can be used to represent and approximate random processes driven by the fractional Brownian motion (a generalization of the method proposed in~\cite{Ryb_DUPU20, RybYus_IOP20}), e.g., the fractional Ornstein--Uhlenbeck process and the fractional Brownian bridge~\cite{Thao_EWJM13}. The obtained results can be further applied to represent and simulate iterated stochastic integrals over the fractional Brownian motion or a combination of independent fractional Brownian motions (a generalization of the method described in~\cite{Ryb_Springer22}).

In comparison with existing studies, the main contribution of this work is as follows:

(1)\;The explicit relations for expansion coefficients of the kernel related to the integral representation of the fractional Brownian motion are derived (these expansion coefficients are defined with respect to the Legendre polynomials).

(2)\;Equations that allow one to calculate exactly the mean square error of the fractional Brownian motion approximation by the polynomial with random coefficients are obtained.

The rest of this paper has the following structure. Section~\ref{secFBM} provides the definition of the fractional Brownian motion and formulates the problem statement. Section~\ref{secPreliminary} contains the relations for integral and spectral representations of the fractional Brownian motion to the extent necessary to obtain the new results and carry out a comparative analysis. The main result of the paper is presented in Section~\ref{secMain}. Section~\ref{secSpFBM} is devoted to the fractional Brownian motion approximation, in which the expressions for the mean square approximation error are obtained. Brief conclusions are given in Section~\ref{secSpConcl}.

\section{The Problem Statement}\label{secFBM}

The paper considers the fractional Brownian motion $B_H(\cdot)$~\cite{Shi_99, BiaHuOksZha_08, Mis_08}, i.e., a centered Gaussian random process with the covariance function
\[
  R_H(t,\tau) = \frac{t^{2H} + \tau^{2H} - |t-\tau|^{2H}}{2}, \ \ \ t,\tau \geqslant 0,
\]
where $H \in (0,1)$ is a parameter called the Hurst index.

The random process $B_H(\cdot)$ can be represented as a stochastic integral over the Brownian motion $B(\cdot) = B_{1/2}(\cdot)$, i.e., the standard Wiener process. Such a representation was first introduced in~\cite{ManNess_SR68}, where $B_H(t)$ is defined as a stochastic integral over the set $(-\infty,t]$. Another representation was proposed in~\cite{DecUst_PA99}, and it underlies the problem being solved:
\begin{equation}\label{eqDefFBMFinite}
  B_H(t) = \int_0^t k_H(t,\tau) dB(\tau).
\end{equation}

On the one hand, the integral representation defined by Equation~\eqref{eqDefFBMFinite} is convenient, since $B_H(t)$ is defined as a stochastic integral over the finite set $[0,t]$. On the other hand, the kernel $k_H(\cdot)$ corresponding to Equation~\eqref{eqDefFBMFinite} is not expressed in a simple way through elementary functions~\cite{BiaHuOksZha_08}:
\[
  k_H(t,\tau) = a_H (t-\tau)^{H-1/2} \, {}_2F_1 \biggl( \frac{1}{2} - H, H - \frac{1}{2}, H + \frac{1}{2}, 1 - \frac{t}{\tau} \biggr) 1(t-\tau),
\]
where the constant $a_H$ is defined as
\[
  a_H = \sqrt{\frac{2H \Gamma(H+1/2) \Gamma(3/2-H)}{\Gamma(2-2H)}}, \ \ \ \text{or for $H \neq 1/2$} \ \ \ a_H = \sqrt{\frac{\pi H (1-2H)}{\Gamma(2-2H) \cos \pi H}},
\]
and the following functions are used: ${}_2F_1(\cdot)$ is the hypergeometric function, $\Gamma(\cdot)$ is the Gamma function, and $1(\cdot)$~is the Heaviside function, i.e., $1(\eta) = 1$ if $\eta > 0$ and $1(\eta) = 0$ if $\eta \leqslant 0$. For $H = 1/2$, we have $a_H = 1$ and $k_H(t,\tau) = 1(t-\tau)$.

The problem to be solved consists of finding expansion coefficients of the function $k_H(\cdot)$ relative to the orthonormal basis of $\LPTT$, the space of square integrable functions defined on $\mathds{T}^2$, where $\mathds{T} = [0,T]$. It is assumed that the basis is formed by all possible products of functions that form the orthonormal basis of $\LPT$, the space of square integrable functions defined on $\mathds{T}$~\cite{Bal_80}. The Legendre polynomials~\cite{SolSemPeshNed_79} are chosen as such a basis:
\begin{equation}\label{eqDefLeg}
  \hat P(i,t) = \sqrt{\frac{2i+1}{T}} \sum\limits_{k=0}^i {l_{ik} \, \frac{t^k}{T^k}}, \ \ \ i = 0,1,2,\dots,
\end{equation}
where
\begin{equation}\label{eqDefLegCoef}
  l_{ik} = (-1)^{i-k} C^i_{i+k} C^{i-k}_i = (-1)^{i-k} \prod\limits_{m=1}^k \frac{(i-m+1)(i+m)}{m^2}.
\end{equation}

The proof that the function $k_H(\cdot)$ belongs to $\LPTT$ follows, e.g., from the representation of the covariance function $R_H(\cdot)$:
\[
  R_H(t,\tau) = \int_\mathds{T} {k_H(t,\theta) k_H(\tau,\theta) d\theta},
\]
which defines a trace class operator in $\LPT$~\cite{GihSco_77}. The It\^o stochastic integral properties~\cite{Oks_03}, in particular the It\^o isometry, allow one to confirm that
\[
  \mathrm{E} \int_\mathds{T} B_H^2(t) dt = \int_\mathds{T} R_H(t,t) dt = \| k_H(\cdot) \|_\LPTT^2,
\]
where $\mathrm{E}$ denotes the mathematical expectation, consequently,
\begin{equation}\label{eqKHSquaredNorm}
  \| k_H(\cdot) \|_\LPTT^2 = \int_\mathds{T} t^{2H} dt = \frac{T^{2H+1}}{2H+1}.
\end{equation}

Thus, it is required to find values
\begin{equation}\label{eqSpKH}
  K_{ij}^H = \int_{\mathds{T}^2} \hat P(i,t) \hat P(j,\tau) k_H(t,\tau) dt d\tau = \int_\mathds{T} \hat P(i,t) \biggl[ \int_0^t k_H(t,\tau) \hat P(j,\tau) d\tau\biggr] dt, \ \ \ i,j = 0,1,2,\dots
\end{equation}

The solution to this problem provides a representation of the fractional Brownian motion as the orthogonal expansion, on the basis of which it is possible to obtain an algorithm for the approximate simulation of its paths. Such an algorithm has two important advantages. First, the fractional Brownian motion approximation in continuous time is constructed, and second, the mean square approximation error can be calculated exactly.

The relations similar to Equation~\eqref{eqSpKH} can be formally written for an arbitrary orthonormal basis of $\LPT$, but it is quite easy to derive analytical expressions for expansion coefficients of the function $k_H(\cdot)$ when choosing the Legendre polynomials. In addition, the Legendre polynomials~\eqref{eqDefLeg} were previously used to represent the Wiener process, which is a special case of the fractional Brownian motion, and such a representation turned out to be quite effective~\cite{Kuz_JVMMF19, Kuz_DUPU20}.

\section{Preliminary Results}\label{secPreliminary}

The function $k_H(\cdot)$ defines the Hilbert--Schmidt operator~\cite{Bal_80} in $\LPT$, namely
\begin{equation}\label{eqLinearIO}
  \ck_H \phi(t) = \int_0^{t} k_H(t,\tau) \phi(\tau) d\tau, \ \ \ \phi(\cdot) \in \LPT,
\end{equation}
and according to~\cite{DecUst_PA99}, this operator is represented as follows:
\begin{equation}\label{eqDefFBMFiniteOper}
  \ck_H = \left\{
    \begin{array}{ll}
      a_H \, \cj_{0+}^{2H} \circ \ca_{f_{1/2-H}} \circ \cj_{0+}^{1/2-H} \circ \ca_{f_{H-1/2}} & \text{for} ~ H < 1/2 \\
      a_H \, \cj_{0+}^1 \circ \ca_{f_{H-1/2}} \circ \cj_{0+}^{H-1/2} \circ \ca_{f_{1/2-H}} & \text{for} ~ H > 1/2,
    \end{array}
  \right.
\end{equation}
where $\ca_{f_\alpha}$ and $\cj_{0+}^\beta$ are the multiplication operator with multiplier $f_\alpha(t) = t^\alpha$, $\alpha > -1/2$, and the left-sided Riemann--Liouville integration operator of fractional order $\beta > 0$, respectively, i.e.,
\[
  \ca_{f_\alpha} \phi(t) = t^\alpha \phi(t), \ \ \ \cj_{0+}^\beta \phi(t) = \frac{1}{\Gamma(\beta)} \int_0^t {\frac{\phi(\tau) d\tau}{(t-\tau)^{1-\beta}}}, \ \ \ t > 0.
\]

For $H = 1/2$, the operator $\ck_H$ coincides with the first-order integration operator $\cj_{0+}^1$.

The representation of the fractional Brownian motion as the expansion into orthogonal series with respect to the orthonormal basis of $\LPT$ is given in~\cite{Ryb_Comp25}. Such a representation relative to the Legendre polynomials~\eqref{eqDefLeg} is presented below:
\begin{equation}\label{eqSpFBMFinite}
  B_H(t) = \sum\limits_{i=0}^\infty \cb_i^H \hat P(i,t), \ \ \ t \in \mathds{T},
\end{equation}
and here
\begin{equation}\label{eqSpFBMCoef}
  \cb_i^H = \sum\limits_{j=0}^\infty K_{ij}^H \cv_j, \ \ \ \cv_i = \int_\mathds{T} {\hat P(i,t) dB(t)}, \ \ \ i = 0,1,2,\dots,
\end{equation}
where values $K_{ij}^H$ are determined by Equation~\eqref{eqSpKH}, and random variables $\cv_i$ are independent and have a standard normal distribution.

Using matrix notation, we obtain $\cb^H = K^H \cv$, where $\cb^H$ and $\cv$ are infinite random column matrices with elements $\cb_i^H$ and $\cv_i$, respectively, and $K^H$ is the infinite matrix with elements $K_{ij}^H$; $K^H$ is the matrix representation of the operator $\ck_H$.

Another result from~\cite{Ryb_Comp25} is the representation of the matrix $K^H$ in the form
\begin{equation}\label{eqSpFBMFiniteOper}
  K^H = \left\{
    \begin{array}{ll}
      a_H \, P^{-2H} A^{1/2-H} P^{-(1/2-H)} A^{H-1/2} & \text{for} ~ H < 1/2 \\
      a_H \, P^{-1} A^{H-1/2} P^{-(H-1/2)} A^{1/2-H} & \text{for} ~ H > 1/2,
    \end{array}
  \right.
\end{equation}
where $A^\alpha$ and $P^{-\beta}$ are infinite matrices with elements $A_{ij}^\alpha$ and $P_{ij}^{-\beta}$, expressions for which are given below; $A^\alpha$ and $P^{-\beta}$ are matrix representations of operators $\ca_{f_\alpha}$ and $\cj_{0+}^\beta$, respectively. For $H = 1/2$, the matrix $K^H$ coincides with the matrix $P^{-1}$ corresponding to the operator $\cj_{0+}^1$.

The paper~\cite{Ryb_Comp25} uses the terminology associated with the spectral form of mathematical description of control systems~\cite{SolSemPeshNed_79}. For instance, infinite column matrices are called spectral characteristics of functions or random processes, including generalized ones; infinite matrices are called two-dimensional spectral characteristics of functions or spectral characteristics of linear operators. Thus, spectral representations of functions, random processes, and linear operators (signals and systems) are considered~\cite{Ryb_Comp25, Ryb_Math23}. The representation of the matrix $K^H$ as a product of four matrices corresponds to a sequential signal transformation (in this case, the input signal is white noise) by linear blocks: two proportional blocks and two fractional-order integral blocks. This approach corresponds to the concept of the spectral form of mathematical description: the use of matrix representations of linear operators corresponding to elementary blocks of control systems to transform expansion coefficients of the input signal~\cite{SolSemPeshNed_79}. In this paper, the specified terminology will not be used for a more concise presentation of the results.

The constructiveness of Equation~\eqref{eqSpFBMFiniteOper} is ensured by analytical expressions for elements of matrices $A^\alpha$ and $P^{-\beta}$ derived in~\cite{Ryb_Comp25}.

Below are expressions for elements of matrices $A^\alpha$ and $P^{-\beta}$ that use the notation $F^\alpha$ for the infinite column matrix formed by expansion coefficients of the power function $f_\alpha(t) = t^\alpha$, $\alpha > -1/2$, relative to the Legendre polynomials~\eqref{eqDefLeg}:
\begin{equation}\label{eqSpFAlphaExplicit}
  \begin{gathered}
    F_i^\alpha = \int_\mathds{T} {t^\alpha \hat P(i,t) dt} = \sqrt{\frac{2i+1}{T}} \sum\limits_{k=0}^i {l_{ik} \, \frac{T^{\alpha+1}}{\alpha+k+1}} = T^\alpha \sqrt{T} \, \frac{\sqrt{2i+1} \, \alpha^{\underline{i}}}{(\alpha + 1)^{\overline{i+1}}}, \ \ \ i = 0,1,2,\dots,
  \end{gathered}
\end{equation}
where $\alpha^{\underline{i}}$ and $(\alpha+1)^{\overline{i+1}}$ are the lower (falling) and upper (rising) factorials, respectively:
\[
  \alpha^{\underline{i}} = \alpha (\alpha-1) \ldots (\alpha-i+1) \ \ \ \text{and} \ \ \ \alpha^{\overline{i}} = \alpha (\alpha+1) \ldots (\alpha+i-1).
\]

In addition to Equation~\eqref{eqSpFAlphaExplicit}, two recurrent relations are obtained~\cite{Ryb_Comp25}, namely by index $i$ and by degree $\alpha$:
\begin{gather}
  F_{i+1}^\alpha = \sqrt{\frac{2i+3}{2i+1}} \, \frac{\alpha-i}{\alpha+i+2} \, F_i^\alpha, \label{eqSpFAlphaImplicit} \\
  F_i^{\alpha+k} = \frac{T^k [(\alpha+1)^{\overline{k}}]^2}{(\alpha-i+1)^{\overline{k}}(\alpha+i+2)^{\overline{k}}} \, F_i^\alpha, \label{eqSpFAlphaImplicitAlpha}
\end{gather}
where $(\alpha+1)^{\overline{k}}$, $(\alpha-i+1)^{\overline{k}}$, and $(\alpha+i+2)^{\overline{k}}$ are upper factorials, and $k$ is the natural number.

For $j$th columns of matrices $A^\alpha$ and $P^{-\beta}$ the following equations hold:
\begin{equation}\label{eqSpAIOperExplicit}
  A_{*j}^\alpha = \sqrt{\frac{2j+1}{T}} \sum\limits_{k=0}^j \frac{l_{jk}}{T^k} \, F^{\alpha+k}, \ \ \ P_{*j}^{-\beta} = \frac{1}{\Gamma(\beta+1)} \sqrt{\frac{2j+1}{T}} \sum\limits_{k=0}^j \frac{l_{jk} k!}{T^k (\beta+1)^{\overline{k}}} \, F^{\beta+k}, \ \ \ j = 0,1,2,\dots,
\end{equation}
where $(\beta+1)^{\overline{k}}$ is the upper factorial.

The derivation of these equations is not difficult, but their use entails computational instability due to peculiarities of machine arithmetic for floating point numbers. In~\cite{Ryb_Comp25}, computationally stable algorithms based on symmetry properties are proposed.

The matrix representation of the multiplication operator relative to an arbitrary orthonormal basis corresponds to the infinite symmetric matrix~\cite{SolSemPeshNed_79}. Using this property as well as Equations~\eqref{eqDefLegCoef}, \eqref{eqSpFAlphaImplicitAlpha}, and~\eqref{eqSpAIOperExplicit}, we~obtain
\begin{equation}\label{eqSpAOperAlphaExplicitSafe1}
  A_{ij}^\alpha = \left\{
    \begin{array}{ll}
      \displaystyle \sqrt{\frac{2j+1}{T}} \, F_i^\alpha \sum\limits_{k=0}^j (-1)^{j-k} \prod\limits_{m=1}^k \biggl( \frac{\alpha+m}{m} \biggr)^2 \frac{j-m+1}{\alpha-i+m} \, \frac{j+m}{\alpha+i+m+1} & \text{for} ~ i \geqslant j \\
      A_{ji}^\alpha & \text{for} ~ i < j,
    \end{array}
  \right.
\end{equation}
or
\begin{equation}\label{eqSpAOperAlphaExplicitSafe2}
  A_{ij}^\alpha = \left\{
    \begin{array}{ll}
      \displaystyle \sqrt{\frac{2j+1}{T}} \, F_i^\alpha \sum\limits_{k=0}^j (-1)^{j-k} \Pi_k^{(ij)} & \text{for} ~ i \geqslant j \\
      A_{ji}^\alpha & \text{for} ~ i < j,
    \end{array}
  \right.
\end{equation}
where
\[
  \Pi_0^{(ij)} = 1, \ \ \ \Pi_k^{(ij)} = \biggl( \frac{\alpha+k}{k} \biggr)^2 \frac{j-k+1}{\alpha-i+k} \, \frac{j+k}{\alpha+i+k+1} \, \Pi_{k-1}^{(ij)}, \ \ \ k = 1,2,\dots
\]

To represent the matrix $P^{-\beta}$ as the sum of symmetric and skew-symmetric matrices, it is sufficient to choose elements $P_{ij}^{-\beta}$ with the even and odd sum of indices, respectively~\cite{Ryb_Comp25}. This property holds for the matrix representation of the fractional integration operator relative to the Legendre polynomials. Taking into account this property as well as Equations~\eqref{eqDefLegCoef},~\eqref{eqSpFAlphaImplicitAlpha}, and~\eqref{eqSpAIOperExplicit}, we have
\begin{equation}\label{eqSpOperIFracLExplicitSafe1}
  P_{ij}^{-\beta} = \left\{
    \begin{array}{ll}
      \displaystyle \frac{1}{\Gamma(\beta+1)} \sqrt{\frac{2j+1}{T}} \, F_i^\beta \sum\limits_{k=0}^j (-1)^{j-k} \prod\limits_{m=1}^k \frac{\beta+m}{m} \, \frac{j-m+1}{\beta-i+m} \, \frac{j+m}{\beta+i+m+1} & \text{for} ~ i \geqslant j \\
      (-1)^{i+j} P_{ji}^{-\beta} & \text{for} ~ i < j,
    \end{array}
  \right.
\end{equation}
or
\begin{equation}\label{eqSpOperIFracLExplicitSafe2}
  P_{ij}^{-\beta} = \left\{
    \begin{array}{ll}
      \displaystyle \frac{1}{\Gamma(\beta+1)} \sqrt{\frac{2j+1}{T}} \, F_i^\beta \sum\limits_{k=0}^j (-1)^{j-k} \Pi_k^{(ij)} & \text{for} ~ i \geqslant j \\
      (-1)^{i+j} P_{ji}^{-\beta} & \text{for} ~ i < j,
    \end{array}
  \right.
\end{equation}
where
\[
  \Pi_0^{(ij)} = 1, \ \ \ \Pi_k^{(ij)} = \frac{\beta+k}{k} \, \frac{j-k+1}{\beta-i+k} \, \frac{j+k}{\beta+i+k+1} \, \Pi_{k-1}^{(ij)}, \ \ \ k = 1,2,\dots, \ \ \ \beta \neq 1,2,3,\dots
\]

Under the condition $\beta = 1$, it is sufficient to use the following result~\cite{SolSemPeshNed_79}:
\begin{equation}\label{eqSpOperIExplicit}
  P_{ij}^{-1} = \left\{
    \begin{array}{ll}
      \displaystyle \frac{T}{2} & \text{for} ~ i = j = 0 \\
      \displaystyle \frac{T}{2\sqrt{4i^2-1}} & \text{for} ~ i = j + 1 \\
      \displaystyle -\frac{T}{2\sqrt{4j^2-1}} & \text{for} ~ j = i + 1 \\
      \displaystyle 0 \vphantom{\frac12} & \text{otherwise},
    \end{array}
  \right.
\end{equation}
and $P^{-\beta} = (P^{-1})^\beta$ for $\beta = 2,3,\dots$

Based on Equations~\eqref{eqSpFAlphaExplicit} and~\eqref{eqSpFAlphaImplicitAlpha}, we can obtain the main result presented in the next section. Equations~\eqref{eqSpAOperAlphaExplicitSafe1} and~\eqref{eqSpOperIFracLExplicitSafe1}, or~\eqref{eqSpAOperAlphaExplicitSafe2} and~\eqref{eqSpOperIFracLExplicitSafe2}, are required for numerical calculations and comparative analysis of the results of this paper and the paper~\cite{Ryb_Comp25}.

\section{The Main Result}\label{secMain}

The expression similar to Equation~\eqref{eqSpAIOperExplicit} can be easily derived for columns of the matrix $K^H$. Based on it, we can obtain the equation for elements of this matrix in expanded form. For this, it is sufficient to use Equation~\eqref{eqDefFBMFiniteOper} as well as the well-known relation~\cite{SamKilMar_87}:
\begin{equation}\label{eqFracInt}
  \cj_{0+}^\beta t^\alpha = \frac{1}{\Gamma(\beta)} \int_0^t {\frac{\tau^\alpha d\tau}{(t-\tau)^{1-\beta}}} = \frac{\Gamma(\alpha+1)}{\Gamma(\alpha+\beta+1)} \, t^{\alpha+\beta}, \ \ \ \alpha > -1.
\end{equation}

\begin{theorem}\label{thmMain}
Let $K^H$ be the infinite matrix with elements $K_{ij}^H$ defined by Equation~\eqref{eqSpKH}, $H \in (0,1)$. Then, for the $j$th column of the matrix $K^H$, we have
\begin{equation}\label{eqSpOperKHLExplicit}
  K_{*j}^H = a_H \Gamma(3/2-H) \sqrt{\frac{2j+1}{T}} \sum\limits_{k=0}^j \frac{l_{jk} (3/2-H)^{\overline{k}}}{T^k (H+1/2+k) k!} \, F^{H+1/2+k}, \ \ \ j = 0,1,2,\dots,
\end{equation}
where $(3/2-H)^{\overline{k}}$ is the upper factorial, and elements of the infinite column matrix $F^{H+1/2+k}$ are given by Equations~\eqref{eqSpFAlphaExplicit}--\eqref{eqSpFAlphaImplicitAlpha}.

For $H \neq 1/2$, values $K_{ij}^H$ satisfy the following equation:
\begin{equation}\label{eqSpOperKHijExplicit}
  K_{ij}^H = a_H \Gamma(1/2-H) \sqrt{\frac{2j+1}{T}} \, F_i^{H+1/2} \sum\limits_{k=0}^j (-1)^{j-k} \, \frac{1/2-H+k}{H+1/2+k} \, \Pi_k^{(ij)},
\end{equation}
where
\begin{gather*}
  \Pi_0^{(ij)} = 1, \ \ \ \Pi_k^{(ij)} = \frac{(H+1/2+k)^2 (k-1/2-H)}{k^3} \, \frac{j-k+1}{H+1/2-i+k} \, \frac{j+k}{H+1/2+i+k+1} \, \Pi_{k-1}^{(ij)}, \\
  k = 1,2,\dots
\end{gather*}
\end{theorem}

\begin{proof}
Expansion coefficients $K_{ij}^H$ are determined by Equation~\eqref{eqSpKH}. According to Equation~\eqref{eqLinearIO}, it can be rewritten as
\begin{align*}
  K_{ij}^H = \int_\mathds{T} \hat P(i,t) \ck_H \hat P(j,t) dt = \sqrt{\frac{2j+1}{T}} \int_\mathds{T} {\hat P(i,t) \sum\limits_{k=0}^j {l_{jk} \, \frac{\ck_H t^k}{T^k}} \, dt}, \ \ \ i,j = 0,1,2,\dots,
\end{align*}
where the operator $\ck_H$ is defined by Equation~\eqref{eqDefFBMFiniteOper}. Further, we consider three cases.

1. The case $H < 1/2$:
\[
  K_{ij}^H = a_H \sqrt{\frac{2j+1}{T}} \int_\mathds{T} {\hat P(i,t) \sum\limits_{k=0}^j {l_{jk} \, \frac{\cj_{0+}^{2H} \circ \ca_{f_{1/2-H}} \circ \cj_{0+}^{1/2-H} \circ \ca_{f_{H-1/2}} t^k}{T^k}} \, dt}.
\]

Based on Equation~\eqref{eqFracInt}, we have
\begin{align*}
  & \cj_{0+}^{2H} \circ \ca_{f_{1/2-H}} \circ \cj_{0+}^{1/2-H} \circ \ca_{f_{H-1/2}} t^k = \cj_{0+}^{2H} \circ \ca_{f_{1/2-H}} \circ \cj_{0+}^{1/2-H} t^{H-1/2+k} \\
  & \ \ \ = \frac{\Gamma(H+1/2+k)}{\Gamma(k+1)} \, \cj_{0+}^{2H} \circ \ca_{f_{1/2-H}} t^k = \frac{\Gamma(H+1/2+k)}{k!} \, \cj_{0+}^{2H} t^{1/2-H+k} \\
  & \ \ \ = \frac{\Gamma(H+1/2+k) \Gamma(3/2-H+k)}{\Gamma(H+3/2+k) k!} \, t^{H+1/2+k} = \frac{\Gamma(3/2-H+k)}{(H+1/2+k) k!} \, t^{H+1/2+k},
\end{align*}
hence,
\begin{align*}
  K_{ij}^H & = a_H \sqrt{\frac{2j+1}{T}} \sum\limits_{k=0}^j \frac{l_{jk}}{T^k} \, \frac{\Gamma(3/2-H+k)}{(H+1/2+k) k!} \int_\mathds{T} {\hat P(i,t) t^{H+1/2+k} dt} \\
  & = a_H \Gamma(3/2-H) \sqrt{\frac{2j+1}{T}} \sum\limits_{k=0}^j \frac{l_{jk}}{T^k} \, \frac{(1/2-H+1) \ldots (1/2-H+k)}{(H+1/2+k) k!} \, F_i^{H+1/2+k} \\
  & = a_H \Gamma(3/2-H) \sqrt{\frac{2j+1}{T}} \sum\limits_{k=0}^j \frac{l_{jk} (3/2-H)^{\overline{k}}}{T^k (H+1/2+k) k!} \, F_i^{H+1/2+k}, \ \ \ i,j = 0,1,2,\dots,
\end{align*}
where $F_i^{H+1/2+k}$ are expansion coefficients of the power function $f_{H+1/2+k}(t) = t^{H+1/2+k}$ relative to the Legendre polynomials~\eqref{eqDefLeg}. This implies Equation~\eqref{eqSpOperKHLExplicit}.

2. The case $H > 1/2$:
\[
  K_{ij}^H = a_H \sqrt{\frac{2j+1}{T}} \int_\mathds{T} {\hat P(i,t) \sum\limits_{k=0}^j {l_{jk} \, \frac{\cj_{0+}^1 \circ \ca_{f_{H-1/2}} \circ \cj_{0+}^{H-1/2} \circ \ca_{f_{1/2-H}} t^k}{T^k}} \, dt}.
\]

Here, we also use Equation~\eqref{eqFracInt}. Since
\begin{align*}
  & \cj_{0+}^1 \circ \ca_{f_{H-1/2}} \circ \cj_{0+}^{H-1/2} \circ \ca_{f_{1/2-H}} t^k = \cj_{0+}^1 \circ \ca_{f_{H-1/2}} \circ \cj_{0+}^{H-1/2} t^{1/2-H+k} \\
  & \ \ \ = \frac{\Gamma(3/2-H+k)}{\Gamma(k+1)} \, \cj_{0+}^1 \circ \ca_{f_{H-1/2}} t^k = \frac{\Gamma(3/2-H+k)}{k!} \, \cj_{0+}^1 t^{H-1/2+k} = \frac{\Gamma(3/2-H+k)}{(H+1/2+k) k!} \, t^{H+1/2+k},
\end{align*}
we again obtain Equation~\eqref{eqSpOperKHLExplicit}.

3. Case $H = 1/2$:
\[
  K_{ij}^H = \sqrt{\frac{2j+1}{T}} \int_\mathds{T} {\hat P(i,t) \sum\limits_{k=0}^j {l_{jk} \, \frac{\cj_{0+}^1 t^k}{T^k}} \, dt}, \ \ \ \cj_{0+}^1 t^k = \frac{t^{k+1}}{k+1}.
\]

The last equality can be rewritten as
\[
  \cj_{0+}^1 t^k = \frac{\Gamma(3/2-H+k)}{(H+1/2+k) k!} \, t^{H+1/2+k},
\]
since $\Gamma(3/2-H+k) = \Gamma(k+1) = k!$ and $H+1/2+k = k+1$, i.e., Equation~\eqref{eqSpOperKHLExplicit} is valid.

Next, we obtain the equation for elements $K_{ij}^H$ in expanded form. Equation~\eqref{eqSpOperKHLExplicit} implies the following relation:
\[
  K_{ij}^H = a_H \Gamma(3/2-H) \sqrt{\frac{2j+1}{T}} \sum\limits_{k=0}^j \frac{l_{jk} (3/2-H)^{\overline{k}}}{T^k (H+1/2+k) k!} \, F_i^{H+1/2+k}, \ \ \ i,j = 0,1,2,\dots,
\]
in which, according to Equations~\eqref{eqDefLegCoef} and~\eqref{eqSpFAlphaImplicitAlpha}, we have
\begin{align*}
  & \frac{l_{jk} (3/2-H)^{\overline{k}}}{T^k (H+1/2+k) k!} \, F_i^{H+1/2+k} \\
  & \ \ \ = (-1)^{j-k} \frac{(j-k+1) \ldots j(j+1) \ldots (j+k) (1/2-H+1) \ldots (1/2-H+k)}{(k!)^3 (H+1/2+k)} \\
  & \ \ \ \ \ \ {} \times \frac{(H+1/2+1)^2 \dots (H+1/2+k)^2}{(H+1/2-i+1) \dots (H+1/2-i+k) (H+1/2+i+2) \dots (H+1/2+i+k+1)} \, F_i^{H+1/2} \\
  & \ \ \ = \frac{(-1)^{j-k} F_i^{H+1/2}}{H+1/2+k} \prod\limits_{m=1}^k \frac{(H+1/2+m)^2 (1/2-H+m)}{m^3} \, \frac{j-m+1}{H+1/2-i+m} \, \frac{j+m}{H+1/2+i+m+1}.
\end{align*}

Using the Gamma function properties, namely $\Gamma(3/2-H) = (1/2-H) \Gamma(1/2-H)$, and also taking into account the equality
\begin{gather*}
  (1/2-H) \prod\limits_{m=1}^k (1/2-H+m) = (1/2-H) \biggl( \, \prod\limits_{m=1}^{k-1} (1/2-H+m) \biggr) (1/2-H+k) \\
  = (1/2-H+k) \prod\limits_{m=0}^{k-1} (1/2-H+m) = (1/2-H+k) \prod\limits_{m=1}^k (m-1/2-H),
\end{gather*}
we obtain
\begin{align*}
  K_{ij}^H & = a_H \Gamma(1/2-H) \sqrt{\frac{2j+1}{T}} \, F_i^{H+1/2} \sum\limits_{k=0}^j (-1)^{j-k} \, \frac{1/2-H+k}{H+1/2+k} \\
  & \ \ \ {} \times \prod\limits_{m=1}^k \frac{(H+1/2+m)^2 (m-1/2-H)}{m^3} \, \frac{j-m+1}{H+1/2-i+m} \, \frac{j+m}{H+1/2+i+m+1},
\end{align*}
and this is equivalent to Equation~\eqref{eqSpOperKHijExplicit}. The theorem is proved.
\end{proof}

\begin{remark}\label{remBM}
If $H = 1/2$, then $K_{ij}^H = P_{ij}^{-1}$, where $P_{ij}^{-1}$ are determined by Equation~\eqref{eqSpOperIExplicit}.
\end{remark}

According to Equation~\eqref{eqSpFAlphaExplicit}, all expansion coefficients $F_i^{H+1/2}$ include the factor $T^{H+1}$, and values $F_i^{H+1/2} / T^{H+1}$ are independent on $T$. Consequently, all expansion coefficients $K_{ij}^H$ include the factor $T^{H+1/2}$, and values $K_{ij}^H / T^{H+1/2}$ are also independent on $T$. This means that it is sufficient to have the matrix representation of the operator $\ck_H$ under the condition $\mathds{T} = [0,1]$. Then, the matrix representation under the condition $\mathds{T} = [0,T]$ can be obtained through the scaling factor $T^{H+1/2}$. Obviously, this property can be proved if we will analyze the dependence of elements $A_{ij}^\alpha$ and $P_{ij}^{-\beta}$ on $T$ and apply Equation~\eqref{eqSpFBMFiniteOper}.

The specified property reflects the self-similarity of the fractional Brownian motion. The factor $1/\sqrt{T}$ in functions $\{\hat P(i,\cdot)\}_{i=0}^\infty$ according to Equation~\eqref{eqDefLeg} leads to the factor $T^H$ on the right-hand side of Equation~\eqref{eqSpFBMFinite}. Recall that for the self-similar random process with the Hurst index $H$, a change in the time scale $t \mapsto at$ is equivalent to a change in the phase scale $x \mapsto a^H x$, i.e., $\mathrm{Law} \bigl( B_H(at), t \geqslant 0 \bigr) = \mathrm{Law} \bigl( a^H B_H(t), t \geqslant 0 \bigr)$, $a > 0$, where $\mathrm{Law}(\cdot)$ means the distribution law~\cite{Shi_99}.

\begin{remark}\label{remOper}
Equations~\eqref{eqSpOperKHLExplicit} and~\eqref{eqSpOperKHijExplicit} are obtained using Equation~\eqref{eqDefFBMFiniteOper}, so in the proof of Theorem~\ref{thmMain}, cases $H < 1/2$, $H > 1/2$ and $H = 1/2$ are considered separately. However, the final result does not require distinguishing these cases. Below, we discuss this in more detail.

Additionally, we define $\cj_{0+}^{-\beta} = \cd_{0+}^\beta$ as the left-sided Riemann--Liouville differentiation operator of fractional order $\beta > 0$, i.e.,
\[
  \cd_{0+}^\beta \phi(t) = \frac{1}{\Gamma(n - \beta)} \, \frac{d^n}{dt^n} \int_0^t {\frac{\phi(\tau) d\tau}{(t-\tau)^{\beta-n+1}}}, \ \ \ t > 0,
\]
where $n = \lfloor \beta \rfloor + 1$ and $\lfloor \,\cdot\, \rfloor$ denotes the integer part. We also define $\cj_{0+}^0$ as the identity operator.

Equation~\eqref{eqFracInt} used in the proof remains valid for an arbitrary $\beta \in (-\infty,\infty)$~\cite{SamKilMar_87}. Moreover, values $K_{ij}^H$ completely determine the kernel $k_H(\cdot)$ and also the operator $\ck_H$, since
\[
  k_H(t,\tau) = \sum\limits_{i,j=0}^\infty K_{ij}^H \hat P(i,t) \hat P(j,t), \ \ \ t,\tau \in \mathds{T}.
\]

This implies that
\[
  \cj_{0+}^{2H} \circ \ca_{f_{1/2-H}} \circ \cj_{0+}^{1/2-H} \circ \ca_{f_{H-1/2}} = \cj_{0+}^1 \circ \ca_{f_{H-1/2}} \circ \cj_{0+}^{H-1/2} \circ \ca_{f_{1/2-H}}, \ \ \ H \in (0,1),
\]
therefore,
\begin{equation}\label{eqDefFBMFiniteOperNew}
  \begin{aligned}
    \ck_H & = a_H \, \cj_{0+}^{2H} \circ \ca_{f_{1/2-H}} \circ \cj_{0+}^{1/2-H} \circ \ca_{f_{H-1/2}} \\
    & = a_H \, \cj_{0+}^1 \circ \ca_{f_{H-1/2}} \circ \cj_{0+}^{H-1/2} \circ \ca_{f_{1/2-H}}, \ \ \ H \in (0,1),
  \end{aligned}
\end{equation}
where the operator $\cj_{0+}^\beta$ is used subject to $\beta \in (-1/2,2)$.
\end{remark}

\section{Approximate Representation of Fractional Brownian Motion}\label{secSpFBM}

In order to approximate the fractional Brownian motion, we replace the series in Equations~\eqref{eqSpFBMFinite} and~\eqref{eqSpFBMCoef} with their partial sums:
\begin{equation}\label{eqSpFBMFiniteApprox}
  B_H(t) \approx \tilde B_H(t) = \sum\limits_{i=0}^{L-1} \tilde \cb_i^H \hat P(i,t), \ \ \ t \in \mathds{T},
\end{equation}
where
\[
  \tilde \cb_i^H = \sum\limits_{j=0}^{L-1} K_{ij}^H \cv_j, \ \ \ i = 0,1,\dots,L-1.
\]

Using matrix notation, we have $\tilde \cb^H = \bar K^H \bar \cv$, where $\tilde \cb^H$ and $\bar \cv$ are random column matrices of size $L$ with elements $\tilde \cb_i^H$ and $\cv_i$, respectively, and $\bar K^H$ is the square matrix of size $L \times L$ with elements $K_{ij}^H$. Independent normally distributed random variables $\cv_i$ are defined by Equation~\eqref{eqSpFBMCoef}.

The notation $\bar M$ corresponds to the truncation of the infinite column matrix or infinite matrix $M$, and the notation $\tilde M$ means that, in addition to the truncation, the remaining elements of $M$ contain some errors, in particular $\bar \cb^H \neq \tilde \cb^H$ and $\bar K^H \neq \tilde K^H$, where the matrix $\tilde K^H$ is defined as follows~\cite{Ryb_Comp25}:
\begin{equation}\label{eqSpFBMFiniteOperApprox}
  \tilde K^H = \left\{
    \begin{array}{ll}
      a_H \, \bar P^{-2H} \bar A^{1/2-H} \bar P^{-(1/2-H)} \bar A^{H-1/2} & \text{for} ~ H < 1/2 \\
      a_H \, \bar P^{-1} \bar A^{H-1/2} \bar P^{-(H-1/2)} \bar A^{1/2-H} & \text{for} ~ H > 1/2,
    \end{array}
  \right.
\end{equation}
and $\bar A^{H-1/2},\bar A^{1/2-H},\bar P^{-2H},\bar P^{-1},\bar P^{-(1/2-H)}$, and $\bar P^{-(H-1/2)}$ are square matrices of size $L \times L$, obtained by the truncation of infinite matrices $A^{H-1/2},A^{1/2-H},P^{-2H},P^{-1},P^{-(1/2-H)}$, and $P^{-(H-1/2)}$, respectively. The inequality $\bar K^H \neq \tilde K^H$ holds under the condition $H \neq 1/2$, and $\tilde K^H = \bar K^H = \bar P^{-1}$ otherwise (see Remark~\ref{remBM}).

The paper~\cite{Ryb_Comp25} proposes the approximate representation of the fractional Brownian motion given by the equation
\begin{equation}\label{eqSpFBMFiniteApproxPlus}
  B_H(t) \approx \tilde B_H^*(t) = \sum\limits_{i=0}^{L-1} \tilde \cb_i^H \hat P(i,t), \ \ \ t \in \mathds{T},
\end{equation}
where
\[
  \tilde \cb_i^H = \sum\limits_{j=0}^{L-1} \tilde K_{ij}^H \cv_j, \ \ \ i = 0,1,\dots,L-1,
\]
i.e., $\tilde \cb^H = \tilde K^H \bar \cv$, where $ \tilde K^H$ is the square matrix of size $L \times L$ with elements $\tilde K_{ij}^H$ corresponding to Equation~\eqref{eqSpFBMFiniteOperApprox}.

The mean square errors of the fractional Brownian motion approximation can be calculated exactly. These errors depend on $H$ and $L$, but such a dependence is not indicated for simplicity:
\[
  \varepsilon = \mathrm{E} \int_\mathds{T} \bigl( B_H(t) - \tilde B_H(t) \bigr)^2 dt, \ \ \ \varepsilon^* = \mathrm{E} \int_\mathds{T} \bigl( B_H(t) - \tilde B_H^*(t) \bigr)^2 dt.
\]

\begin{theorem}\label{thmApprox}
The mean square approximation errors $\varepsilon$ and $\varepsilon^*$ satisfy the following expressions:
\[
  \varepsilon = \| k_H(\cdot) \|_\LPTT^2 - \|\bar K^H\|^2, \ \ \ \varepsilon^* = \| k_H(\cdot) \|_\LPTT^2 - \|\bar K^H\|^2 + \|\bar K^H - \tilde K^H\|^2,
\]
where $\| k_H(\cdot) \|_\LPTT^2$ is given by Equation~\eqref{eqKHSquaredNorm}, and the notation $\|\cdot\|$ means the Euclidean norm of the matrix.
\end{theorem}

\begin{proof}
Equations~\eqref{eqSpFBMFinite} and~\eqref{eqSpFBMFiniteApprox} can be rewritten as
\[
  B_H(t) = \sum\limits_{i,j=0}^\infty K_{ij}^H \hat P(i,t) \cv_j, \ \ \ \tilde B_H(t) = \sum\limits_{i,j=0}^{L-1} K_{ij}^H \hat P(i,t) \cv_j,
\]
then
\[
  B_H(t) - \tilde B_H(t) = \sum\limits_{i,j=0}^\infty K_{ij}^H \hat P(i,t) \cv_j - \sum\limits_{i,j=0}^{L-1} K_{ij}^H \hat P(i,t) \cv_j = \sum\limits_{(i,j) \in \mathrm{Y}} K_{ij}^H \hat P(i,t) \cv_j,
\]
where $\mathrm{Y} = (\mathrm{N}_0 \times \mathrm{N}_L) \cup (\mathrm{N}_L \times \mathrm{N}_0)$, $\mathrm{N}_k = \{k,k+1,k+2,\dots\}$.

Functions $\{\hat P(i,\cdot)\}_{i=0}^\infty$ form the orthonormal system in $\LPT$, and random variables $\{\cv_j\}_{j=0}^\infty$ form the orthonormal system in $\cl_2$, the space of random variables~\cite{GihSco_77}. Using properties of orthogonal expansions, we obtain
\[
  \varepsilon = \mathrm{E} \int_\mathds{T} \biggl( \, \sum\limits_{(i,j) \in \mathrm{Y}}^\infty K_{ij}^H \hat P(i,t) \cv_j \biggr)^2 dt =
  \sum\limits_{(i,j) \in \mathrm{Y}}^\infty (K_{ij}^H)^2 = \sum\limits_{i,j=0}^\infty (K_{ij}^H)^2 - \sum\limits_{i,j=0}^{L-1} (K_{ij}^H)^2.
\]

According to Parseval's identity,
\[
  \| k_H(\cdot) \|_\LPTT^2 = \sum\limits_{i,j=0}^\infty (K_{ij}^H)^2,
\]
hence, $\varepsilon = \| k_H(\cdot) \|_\LPTT^2 - \|\bar K^H\|^2$.

Next, we can rewrite Equation~\eqref{eqSpFBMFiniteApproxPlus} in the form
\[
  \tilde B_H^*(t) = \sum\limits_{i,j=0}^{L-1} \tilde K_{ij}^H \hat P(i,t) \cv_j,
\]
then
\[
  B_H(t) - \tilde B_H^*(t) = B_H(t) - B_H(t) + B_H(t) - \tilde B_H^*(t) = \sum\limits_{(i,j) \in \mathrm{Y}} K_{ij}^H \hat P(i,t) \cv_j + \sum\limits_{i,j=0}^{L-1} (K_{ij}^H - \tilde K_{ij}^H) \hat P(i,t) \cv_j.
\]

Using properties of orthogonal expansions, we have
\begin{align*}
  \varepsilon^* & = \mathrm{E} \int_\mathds{T} \biggl( \, \sum\limits_{(i,j) \in \mathrm{Y}} K_{ij}^H \hat P(i,t) \cv_j + \sum\limits_{i,j=0}^{L-1} (K_{ij}^H - \tilde K_{ij}^H) \hat P(i,t) \cv_j \biggr)^2 dt \\
  & = \sum\limits_{(i,j) \in \mathrm{Y}}^\infty (K_{ij}^H)^2 + \sum\limits_{i,j=0}^{L-1} (K_{ij}^H - \tilde K_{ij}^H)^2 = \| k_H(\cdot) \|_\LPTT^2 - \|\bar K^H\|^2 + \|\bar K^H - \tilde K^H\|^2.
\end{align*}
The theorem is proved.
\end{proof}

The mean square approximation errors $\varepsilon$ and $\varepsilon^*$ are related as $\varepsilon^* > \varepsilon$, except for the case $H = 1/2$, for which $\varepsilon^* = \varepsilon$, since $\varepsilon^* = \varepsilon + \|\bar K^H - \tilde K^H\|^2$.

In~\cite{Ryb_Comp25}, the norm of the difference between covariance functions of random processes $B_H(\cdot)$ and $\tilde B_H(\cdot)$ was chosen as the accuracy measure of the approximation instead of the mean square approximation error of the random process: Equation~\eqref{eqSpFBMFiniteApprox} ensures that the mathematical expectation of the random process $\tilde B_H(\cdot)$ is equal to zero for an arbitrary $L$, but its covariance function differs from $R_H(\cdot)$.

The mean square approximation errors $\varepsilon$ and $\varepsilon^*$ are given in Tables~\ref{tabBHError} and~\ref{tabBHErrorAsterisk} for $T = 1$, $L = 4,8,\dots,1024$, and $H = 0.1,0.2,\dots,0.9$.

\begin{table}[ht]
\begin{center}
\renewcommand{\arraystretch}{1.1}
\caption{The mean square approximation error $\varepsilon$}\label{tabBHError}
\begin{tabular}{cccccccccc}
  \hline
  $H$ & $L = 4$ & $L = 8$ & $L = 16$ & $L = 32$ & $L = 64$ & $L = 128$ & $L = 256$ & $L = 512$ & $L = 1024$ \\
  \hline
  0.1 & 0.384241 & 0.322871 & 0.271951 & 0.229895 & 0.195015 & 0.165934 & 0.141569 & 0.121067 & 0.103751 \\
  0.2 & 0.186574 & 0.136214 & 0.100394 & 0.074562 & 0.055684 & 0.041750 & 0.031391 & 0.023650 & 0.017844 \\
  0.3 & 0.103451 & 0.065528 & 0.042250 & 0.027513 & 0.018016 & 0.011834 & 0.007788 & 0.005130 & 0.003381 \\
  0.4 & 0.060670 & 0.033037 & 0.018487 & 0.010481 & 0.005981 & 0.003424 & 0.001963 & 0.001127 & 0.000647 \\
  0.5 & 0.035714 & 0.016667 & 0.008065 & 0.003968 & 0.001969 & 0.000980 & 0.000489 & 0.000244 & 0.000122 \\
  0.6 & 0.020455 & 0.008205 & 0.003434 & 0.001466 & 0.000632 & 0.000273 & 0.000119 & 0.000052 & 0.000022 \\
  0.7 & 0.013216 & 0.004937 & 0.001924 & 0.000763 & 0.000305 & 0.000123 & 0.000050 & 0.000020 & 0.000008 \\
  0.8 & 0.021488 & 0.011508 & 0.006394 & 0.003602 & 0.002043 & 0.001164 & 0.000666 & 0.000381 & 0.000219 \\
  0.9 & 0.081197 & 0.061740 & 0.046942 & 0.035625 & 0.027012 & 0.020475 & 0.015518 & 0.011760 & 0.008913 \\
  \hline
\end{tabular}
\end{center}
\end{table}

For $\mathds{T} = [0,T]$, it is sufficient to multiply the errors given in Tables~\ref{tabBHError} and~\ref{tabBHErrorAsterisk} by $T^{2H+1}$, since this factor is present in all the error terms: $\| k_H(\cdot) \|_\LPTT^2$, $\|\bar K^H\|^2$, and $\|\bar K^H - \tilde K^H\|^2$. This is another consequence of the self-similarity of the fractional Brownian motion.

\begin{table}[ht]
\begin{center}
\renewcommand{\arraystretch}{1.1}
\caption{The mean square approximation error $\varepsilon^*$}\label{tabBHErrorAsterisk}
\begin{tabular}{cccccccccc}
  \hline
  $H$ & $L = 4$ & $L = 8$ & $L = 16$ & $L = 32$ & $L = 64$ & $L = 128$ & $L = 256$ & $L = 512$ & $L = 1024$ \\
  \hline
  0.1 & 0.385941 & 0.324399 & 0.273198 & 0.230844 & 0.195706 & 0.166421 & 0.141905 & 0.121296 & 0.103905 \\
  0.2 & 0.186654 & 0.136295 & 0.100453 & 0.074599 & 0.055705 & 0.041761 & 0.031396 & 0.023653 & 0.017846 \\
  0.3 & 0.103540 & 0.065560 & 0.042260 & 0.027516 & 0.018017 & 0.011835 & 0.007788 & 0.005130 & 0.003381 \\
  0.4 & 0.060718 & 0.033052 & 0.018491 & 0.010483 & 0.005981 & 0.003424 & 0.001963 & 0.001127 & 0.000647 \\
  0.5 & 0.035714 & 0.016667 & 0.008065 & 0.003968 & 0.001969 & 0.000980 & 0.000489 & 0.000244 & 0.000122 \\
  0.6 & 0.020484 & 0.008212 & 0.003436 & 0.001466 & 0.000632 & 0.000273 & 0.000119 & 0.000052 & 0.000022 \\
  0.7 & 0.013274 & 0.004949 & 0.001927 & 0.000764 & 0.000305 & 0.000123 & 0.000050 & 0.000020 & 0.000008 \\
  0.8 & 0.021554 & 0.011519 & 0.006396 & 0.003602 & 0.002043 & 0.001164 & 0.000666 & 0.000381 & 0.000219 \\
  0.9 & 0.081270 & 0.061755 & 0.046945 & 0.035625 & 0.027012 & 0.020475 & 0.015518 & 0.011760 & 0.008913 \\
  \hline
\end{tabular}
\end{center}
\end{table}

On a qualitative level, these results are close to those published in~\cite{Ryb_Comp25}. Here, we mean the convergence rate depending on $H$ that is defined by the analysis of the above tables, although another accuracy measure of the approximation was used in~\cite{Ryb_Comp25}. However, it is interesting to note that even for small $L$, the data from Tables~\ref{tabBHError} and~\ref{tabBHErrorAsterisk} are very close. This implies that the representation of the matrix $K^H$ as a product of four matrices has almost no effect on the accuracy for the approximate representation but allows one to form computationally stable algorithms (the numerical results presented in this paper were obtained using C++ program\footnote{The source code is available at \url{https://github.com/rkoffice/fBm}} with Boost.Multiprecision library). For Hurst indices close to the boundaries of the interval $(0,1)$, the high relative error and low convergence rate are largely due to properties of the function $k_H(\cdot)$. To calculate the relative error, it is sufficient to compare the data from Tables~\ref{tabBHError} and~\ref{tabBHErrorAsterisk} with the squared norm of the function $k_H(\cdot)$, and this squared norm is equal to $1/(2H+1)$ for $T = 1$.

The operator $\cj_{0+}^\beta$ subject to $\beta \in (-1/2,2)$ is discussed in Remark~\ref{remOper}. For this operator, we can use the matrix representation $P^{-\beta}$. Elements $P_{ij}^{-\beta}$ determined by Equations~\eqref{eqSpAOperAlphaExplicitSafe1} and~\eqref{eqSpAOperAlphaExplicitSafe2} are well-defined not only for non-integer $\beta > 0$ but also for $\beta \in (-1/2,0)$. Recall that elements $P_{ij}^{-1}$ are given by the formula~\eqref{eqSpOperIExplicit}, and $\cj_{0+}^0$ is the identity operator, so $P_{ij}^0 = \delta_{ij}$ is the Kronecker delta.

Consequently, the following expressions can be used to approximate the fractional Brownian motion by Equation~\eqref{eqSpFBMFiniteApproxPlus}:
\begin{equation}\label{eqSpFBMFiniteOperApproxA}
  \tilde K^H = a_H \, \bar P^{-2H} \bar A^{1/2-H} \bar P^{-(1/2-H)} \bar A^{H-1/2}, \ \ \ H \in (0,1),
\end{equation}
or
\begin{equation}\label{eqSpFBMFiniteOperApproxB}
  \tilde K^H = a_H \, \bar P^{-1} \bar A^{H-1/2} \bar P^{-(H-1/2)} \bar A^{1/2-H}, \ \ \ H \in (0,1),
\end{equation}
instead of Equation~\eqref{eqSpFBMFiniteOperApprox}.

The mean square approximation error $\varepsilon^*$ can be calculated exactly by the expression proved in Theorem~\ref{thmApprox}. The results obtained by Equation~\eqref{eqSpFBMFiniteOperApproxA} for $H = 0.6,0.7,0.8,0.9$, as well as the results obtained by Equation~\eqref{eqSpFBMFiniteOperApproxB} for $H = 0.1,0.2,0.3,0.4$, are given in Table~\ref{tabBHErrorAsteriskNew} for $T = 1$ and $L = 4,8,\dots,1024$. For $H = 1/2$, Equations~\eqref{eqSpFBMFiniteOperApproxA} and~\eqref{eqSpFBMFiniteOperApproxB} provide the same result. It is easy to see that the data from Tables~\ref{tabBHErrorAsterisk} and~\ref{tabBHErrorAsteriskNew} differ insignificantly.

\begin{table}[ht]
\begin{center}
\renewcommand{\arraystretch}{1.1}
\caption{The mean square approximation error $\varepsilon^{*}$}\label{tabBHErrorAsteriskNew}
\begin{tabular}{cccccccccc}
  \hline
  $H$ & $L = 4$ & $L = 8$ & $L = 16$ & $L = 32$ & $L = 64$ & $L = 128$ & $L = 256$ & $L = 512$ & $L = 1024$ \\
  \hline
  0.1 & 0.387505 & 0.324598 & 0.272852 & 0.230361 & 0.195257 & 0.166061 & 0.141638 & 0.121105 & 0.103773 \\
  0.2 & 0.188157 & 0.136902 & 0.100683 & 0.074680 & 0.055731 & 0.041768 & 0.031398 & 0.023653 & 0.017845 \\
  0.3 & 0.103962 & 0.065718 & 0.042318 & 0.027537 & 0.018024 & 0.011837 & 0.007788 & 0.005130 & 0.003381 \\
  0.4 & 0.060756 & 0.033064 & 0.018495 & 0.010484 & 0.005982 & 0.003424 & 0.001963 & 0.001127 & 0.000647 \\
  0.5 & 0.035714 & 0.016667 & 0.008065 & 0.003968 & 0.001969 & 0.000980 & 0.000489 & 0.000244 & 0.000122 \\
  0.6 & 0.020497 & 0.008216 & 0.003437 & 0.001466 & 0.000632 & 0.000273 & 0.000119 & 0.000052 & 0.000022 \\
  0.7 & 0.013329 & 0.004961 & 0.001930 & 0.000764 & 0.000305 & 0.000123 & 0.000050 & 0.000020 & 0.000008 \\
  0.8 & 0.021687 & 0.011556 & 0.006409 & 0.003607 & 0.002045 & 0.001165 & 0.000666 & 0.000381 & 0.000219 \\
  0.9 & 0.081701 & 0.061995 & 0.047097 & 0.035724 & 0.027077 & 0.020517 & 0.015546 & 0.011779 & 0.008925 \\
  \hline
\end{tabular}
\end{center}
\end{table}

\begin{remark}
Next, we briefly describe the reason for the low convergence rate when $H$ close to 0 or 1. In Section~\ref{secPreliminary}, expressions for expansion coefficients $F_i^\alpha$ of the power function $f_\alpha(t) = t^\alpha$ relative to the Legendre polynomials~\eqref{eqDefLeg} are given. For them, according to Parseval's identity, we have
\[
  \int_\mathds{T} f_\alpha^2(t) dt = \sum\limits_{i=0}^\infty (F_i^\alpha)^2.
\]

We study the convergence rate for the series from this identity for non-integer $\alpha$ (for integer $\alpha$, only the first $\alpha+1$ expansion coefficients $F_i^\alpha$ are non-zero). Applying the Raabe--Duhamel test~\cite{PolMan_07}, we obtain the equality
\[
  \lim\limits_{i \to \infty} i \biggl( \frac{(F_i^\alpha)^2}{(F_{i+1}^\alpha)^2} - 1 \biggr) = \lim\limits_{i \to \infty} i \biggl( \frac{2i+1}{2i+3} \biggl[ \frac{\alpha+i+2}{\alpha-i} \biggr]^2 - 1 \biggr) = 4\alpha+3,
\]
which takes into account Equation~\eqref{eqSpFAlphaImplicit}. The convergence of the series under consideration is equivalent to the convergence of Dirichlet series
\[
  \sum\limits_{i=1}^\infty \frac{1}{i^{4\alpha+3}},
\]
which converges in the same way as the integral
\[
  \int_1^\infty \frac{dt}{t^{4\alpha+3}} = -\frac{1}{(4\alpha+2) \, t^{4\alpha+2}} \bigg|_1^\infty = -\frac{1}{4\alpha+2} \biggl( \lim\limits_{t \to \infty} \frac{1}{t^{4\alpha+2}} - 1 \biggr),
\]
i.e., we have the following condition for the convergence: $2\alpha+1 > 0$. Moreover, it follows that
\[
  \biggl\| f_\alpha(\cdot) - \sum\limits_{i=0}^{L-1} F_i^\alpha \hat P(i,\cdot) \biggl\|_\LPT^2 = \sum\limits_{i=L}^\infty (F_i^\alpha)^2 \approx \frac{C^2}{L^{4\alpha+2}}
  \ \ \ \text{and} \ \ \
  \biggl\| f_\alpha(\cdot) - \sum\limits_{i=0}^{L-1} F_i^\alpha \hat P(i,\cdot) \biggl\|_\LPT \approx \frac{C}{L^{2\alpha+1}},
\]
where $C > 0$ is a constant independent of $L$.

A more precise result can be obtained based on the estimate given in~\cite{HouSchwabSuli_JNA02, CanHusQuaZan_06}, but in this context the only important thing is that for $\alpha \to -1/2$ the convergence rate tends to zero.

Various representations of the function $k_H(\cdot)$ contain the factor $(t/\tau)^{H-1/2} = t^{H-1/2} \, \tau^{1/2-H}$~\cite{BiaHuOksZha_08}. There is also the following estimate for this function~\cite{DecUst_PA99}:
\[
  |k_H(t,\tau)| \leqslant C_H \tau^{-|H-1/2|} (t-\tau)^{-\max\{1/2-H,0\}} 1(t-\tau),
\]
where $C_H > 0$ is a constant depending on $H$.

The function $k_H(\cdot)$ is not an elementary function. Moreover, it is a function of two variables and therefore the study of the convergence rate for the series
\[
  \sum\limits_{i,j=0}^\infty (K_{ij}^H)^2
\]
is a separate problem. But factors $t^{H-1/2}$ and $\tau^{1/2-H}$ in various representations of the function $k_H(\cdot)$ imply that the convergence rate for this series tends to zero for $H \to 0$ or $H \to 1$ (the equivalent condition $-|H-1/2| \to -1/2$).
\end{remark}

\section{Conclusions}\label{secSpConcl}

The paper considers the orthogonal expansion of the fractional Brownian motion relative to the Legendre polynomials. This expansion underlies the method for approximation and simulation of this random process in continuous time, namely the fractional Brownian motion is approximately represented by the polynomial with random coefficients. The relations for expansion coefficients of the kernel related to the integral representation of the fractional Brownian motion and equations for the mean square approximation error are obtained. The proposed method can be used not only to simulate the paths of the considered random process but also other random processes as well as functionals depending on the fractional Brownian motion.

\end{document}